\documentclass{article}
\usepackage[utf8]{inputenc}

\usepackage[margin=1in]{geometry}
\usepackage{amsfonts}
\usepackage{amsthm}
\usepackage{amsmath}
\usepackage{mathtools}
\usepackage{amssymb}
\usepackage{siunitx}
\usepackage{physics}
\usepackage{enumerate}
\usepackage{enumitem}
\usepackage{multicol}
\usepackage[noadjust]{cite}
\usepackage{caption}
\usepackage{subcaption}
\usepackage{graphicx}
\usepackage{tcolorbox}
\usepackage{dsfont}
\usepackage{xcolor}
\usepackage{tikz-cd}
\usetikzlibrary{backgrounds}
\usepackage{subfiles}
\usepackage{longtable}
\usepackage{hyperref}
\usepackage{bbm}
\makeatletter
\def\blfootnote{\gdef\@thefnmark{}\@footnotetext}
\makeatother
\newtheorem{theorem}{Theorem}
\newtheorem*{theorem*}{Theorem}

\newtheorem{lemma}[theorem]{Lemma}

\newtheorem{propn}[theorem]{Proposition}

\newcommand{\B}[1]{\mathbb{#1}}
\newcommand{\C}[1]{\mathcal{#1}}

\newcommand{\is}{\mathbbm{1}}

\usepackage[skip=7.5pt plus1pt, indent=20pt]{parskip}

\title{Branching-Selection Particle Systems and Inverse First Passage Problems}
\author{Jacob Mercer\footnote{\texttt{jacob.mercer@maths.ox.ac.uk}, Department of Statistics, University of Oxford}}
\date{}  

\begin{document}
\maketitle

\begin{abstract} 
   A generalised inverse first passage problem asks whether, given a probability measure $p$ on $[0,\infty]$, one can find a boundary $b:[0,\infty]\to \B{R}$ such that the stopping time:
   \begin{align*}
       \tau:=\inf\left\{t:\Lambda\int_0^t \omega(W_s-b(s))ds \geq U\right\}
   \end{align*}
   has distribution $p$, where $U\sim Exp(1)$, $\Lambda\in(0,\infty)$ and $\omega$ is a monotonic decreasing function. We construct a branching-selection particle system whose hydrodynamic limit is governed by a free boundary problem and connect this to the generalised inverse first passage problem. In the $N$-particle system, particles move as independent Brownian motions, branch at a prescribed rate, and are removed at a rate proportional to their location relative to a position $b^N(t)$ which is a function of the empirical distribution. We identify the limit of $b^N$ as the solution of the inverse first passage problem.
\end{abstract}

\section{Introduction}\blfootnote{This publication is based on work supported by the EPSRC Centre for Doctoral Training in Mathematics of Random Systems: Analysis, Modelling, and Simulation (EP/S023925/1)}
The inverse first passage problem (IFPP) is a classical problem in probability which originates from a question posed by Shiryaev. Given a standard Brownian motion $(W_t)_{t\geq 0}$ and a probability distribution $p$ on $[0,\infty]$, does there exist a boundary $b(t)$ such that the first passage time $\tau:=\inf\{t:W_t\leq b(t)\}$ has distribution $\tau\sim p$. The problem gained interest as it relates to credit default; defaulting being modelled as the point at which the stochastic process crosses the boundary $b$, and $p$ being the distribution of default time, which may be observed from the price of credit-default options (see \cite{allevanedaZhu}, \cite{hullWhiteI}, \cite{hullWhiteII}, for example). 

%The problem has recieved much attention in the literature. Anulova \cite{anulova} establishes that, in the two-sided case, when the stopping time is defined as $\tau:=\inf\{t:|W_t|\geq b(t)\}$ then there indeed exists a barrier $b$ such that $\tau\sim p$ for any probability measure $p$ on $(0,\infty]$. It was subsequently shown by Chen et al. \cite{ccs22}, \cite{cccs11}, \cite{cccs06} that, for general stochastic processes $X$ defined by $dX_t = \mu(X_t,t)dt + \sigma(X_t,t)dW_t$, then there exists, under certain conditions, a barrier solving the IFPP, and the barrier inherits regularity from the regularity of $p$. %In particular, they showed that if $p((t,\infty])\in C^{\alpha+1/2}((t_1,t_2))$, where $1/2<\alpha \notin \B{N}$, then $b(t)\in C^\alpha((t_1,t_2))$. Janson and Ekstr\"{o}m also connected the IFPP to a certain optimal stopping problem \cite{ekstromJansen}. 

It was noted by Klump \cite{ifptpAsHydroKlump} that, in the special case that $p$ is the Exponential distribution $Exp(1)$, the IFPP can be related to the hydrodynamic limit of the $N$-branching Brownian motion particle system ($N$-BBM). In the $N$-BBM, $N$ particles move on $\B{R}$ as independent Brownian motions, each branching independently at rate $1$. Simultaneously with branching, the leftmost particle is deleted, so that there remains exactly $N$ particles at all times. It was shown by De Masi et al. \cite{hydroNBBM} that as $N\to\infty$, the empirical density of the $N$-BBM converges to the solution of the free boundary problem:
\begin{align*}
\begin{cases}
    u_t=\frac{1}{2}u_{xx} + u  & x\in(L_t,\infty), t>0, \\
    u(x,t)=0 & x\notin (L_t,\infty), t>0, \\
    \int_{L_t}^\infty u(y,t)dy=1 & t>0. \\
\end{cases}
\end{align*}
By the Feynman-Kac theorem, we can understand the solution $u$ by $\int_x^\infty u(y,t)dy = e^t\B{P}(W_t\geq x, \tau>t)$, where $W$ is a Brownian motion with initial density $u(x,0)$ and $\tau:=\inf\{t:W_t\leq L_t\}$. Then taking $x\to -\infty$, the condition $\int_{L_t}^\infty u(y,t)dy=1$ gives that $\B{P}(\tau>t)=e^{-t}$. That is, if $(u,L)$ solves the free boundary problem, then $L$ solves the inverse first passage problem with $p=Exp(1)$. Since the location of the leftmost particle of the $N$-BBM ($X_1^N(t)$, say) converges to the free boundary of the PDE, we can think of $X_1^N(t)$ as approximating the solution of the IFPP. As noted by Klump \cite{ifptpAsHydroKlump}, to obtain a solution to the general IFPP, one must simply consider the $N$-BBM with time dependent branching rate $f(t)$, where $p([t,\infty))=\exp(-\int_0^t f(s)ds)$. Under this small adaptation, the $N$-BBM should have a hydrodynamic limit which solves the PDE $u_t(x,t)= \frac12u_{xx}(x,t)+f(t)u(x,t)$
on $(L_t,\infty)$; correspondingly, if $(u,L)$ solves this PDE, then $L$ solves the IFPP with $p((t,\infty)=\exp(-\int_0^t f(s)ds)$. 

In \cite{ettingerHeningEvans}, \cite{ettingerHeningKwong}, a generalisation of the IFPP is studied. They consider the stopping time:
\begin{align}
    \tau:=\inf\Big\{t:\Lambda \int_0^t \omega(W_s-b(s))ds > U\Big\},
\end{align}
where $\omega$ is a decreasing function with $\omega(-\infty)=1$ and $\omega(\infty)=0$ and $U\sim Exp(1)$. Thinking of $\tau$ as a default time, this generalisation considers that defaulting does not occur at the immediate instant that $W_t\leq b(t)$, but as a function of the path of $W$ up to time $t$, and is used as a model of defaulting (see \cite{smoothDefault}, for example). If we consider $\omega(x)=\is_{x\leq 0}$, this is the first time that $W$ has spent more than $Exp(1)/\Lambda$ amount of time below $b$. Then the soft-killed inverse first passage problem (SKIFPP) is the problem of finding $b$ such that $\tau\sim p$ given a distribution $p$ and an initial distribution for $W_0$. It is already known that, with some additional assumptions on the initial density and on $p$, the SKIFPP has a unique solution $b$ when $\omega(x)=\is_{x\leq 0}$ (\cite{ettingerHeningKwong}) or $\omega\in C^3(\B{R})$ with $\{x:\omega(x)\in (0,1)\}$ bounded (\cite{ettingerHeningEvans}). Moreover, the sub-probability density of $\{W_t,\tau>t\}$ is described by the free boundary problem:
\begin{align*}\begin{cases}
    u_t(x,t)=\frac{1}{2}u_{xx}(x,t) - \omega(x-b(t))u(x,t) & x\in \B{R}, t>0, \\
    u(x,0)=\rho(x) & x\in \B{R}\\
    \int_\B{R} u(y,t)dy=p([t,\infty))\\
    %\int_{\B{R}}\omega(y-b(0))u(y,0)dy = f(0).
\end{cases}\end{align*}
The free boundary $b$ is determined by the integral constraint $\int_{\B{R}}u(y,t)dy=p([t,\infty))$.

In this paper, we find the branching-selection particle system whose hydrodynamic limit corresponds to the solution of this SKIFPP. %That is, we find the $N$ particle system corresponding to the SKIFPP as the $N$-BBM corresponds to the IFPP. 
In particular, we will construct a branching-selection particle system whose empirical density converges to the solution of the free boundary problem:
\begin{align} \label{omegaHydroLim}
    \begin{cases}
    u_t(x,t)=\frac{1}{2}u_{xx}(x,t) + f(t)u(x,t) - \Lambda\omega(x-b(t))u(x,t) & x\in \B{R}, t>0, \\
    u(x,0)=\rho(x) & x\in \B{R}\\
    \int_\B{R} u(y,t)dy=1\\
    %\int_{\B{R}}\omega(y-b(0))u(y,0)dy = f(0)
\end{cases}
\end{align}
and show that if $(u,b)$ is a solution, then $b$ solves the SKIFPP with initial distribution $\rho$ and $p([t,\infty))=\exp(-\int_0^t f(s)ds)$. In the particle system, particles branch at rate $f$ and a particle at location $x$ is deleted with probability proportional to $\omega(x-b^N(t))$ where $b^N(t)$ is a function of the empirical distribution at time $t$. Accordingly, $b^N(t)$ can be seen as an approximation of the solution $b$ of the SKIFPP. 

In the specific case that $\omega(x)=\is_{x\leq 0}$, if we choose $f(t)\equiv 1$ and $\Lambda\in (1,\infty)$, then we can combine the PDE and the integral condition to yield the non-linear PDE
\begin{align*}
    u_t(x,t)=\frac12 u_{xx}(x,t)+u(x,t)\left(1-\Lambda\is_{\int_{-\infty}^x u(y,t)dy \leq \Lambda^{-1}}\right).
\end{align*}
In this case, $b(t)$ is the point left of which a proportion $\Lambda^{-1}$ of the mass of $u$ lies. %In terms of the particle system, $(\is_{x\leq 0},1,\alpha^{-1},N)$-BBM is the branching Brownian motion in which at each branching event, we choose a particle uniformly from among the $\lfloor \alpha N\rfloor$ leftmost to be deleted. 

\section{Construction of the process}\label{constructionSec}

Essentially this process will be a branching Brownian motion in which a particle at location $X(t)$ is deleted at rate $\Lambda\omega(X(t)-b^N(t))$ where $b^N:[0,\infty)\to \B{R}$ is a reference point chosen so that the total branching rate at time $t$ is $Nf(t)$. Branching and deletion will happen simultaneously so that there are always exaclty $N$ particles. 

We now give a formal construction. Throughout this work, we will assume that $\omega:\B{R}\to [0,1]$, the \textit{selection weighting}, is a monotone decreasing function with $\omega(x)\to 1$ as $x\to-\infty$ and $\omega(x)\to 0$ as $x\to\infty$. We also assume that $f:[0,\infty)\to [0,\infty)$, the \textit{branching rate} function,is bounded and that $\Lambda > \|f\|_\infty$. 

Then we will describe the $(\omega,f,\Lambda,N)$-BBM at time $t$ by the probability measure $\mu_t^N:=\frac1N\sum_{i=1}^N \delta_{X_i^N(t)}$ where $X^N(t):=(X_1^N(t),\ldots,X_N^N(t))$ denote the locations of the $N$ particles in the system. Let $X^N(0)\in \B{R}^N$ be some initial configuration with $X_1^N(0)\leq X_2^N(0)\leq \cdots\leq X^N_N(0)$ and define $\rho^N:=\frac1N\sum_{i=1}^N \delta_{X^N_i(0)}$. Suppose that $\rho^N$ converges weakly to $\rho$, with $\rho\ll\lambda$ ($\rho$ is absolutely continuous with respect to the Lebesgue measure $\lambda$). We will use $\rho$ to denote both the probability distribution and its density function.

Let $\C{N}(t)$ be a Poisson process of intensity $Nf(t)$ and let the discontinuities of $\C{N}(t)$ be $\tau_1<\tau_2<\cdots$. These will be the branching times of the process. We construct the process inductively as follows. Initially, the $N$ particles have configuration $X^N(0)$. Then on $[0,\tau_1)$, drive the particle at location $X_i^N(0)$ by Brownian motion $(W_i(t))_{0\leq s<\tau_1}$, so the locations of the $N$ particles at time $t\in [0,\tau_1)$ are $X_i^N(t):=X_i^N(0)+W_i(t)$ for $i\in [N]$, where $[N]$ denotes the set $\{1,2,\ldots,N\}$. At each stopping time $\tau_i$, we will relabel the particles so that they are ordered. That is $X_1^N(\tau_i)\leq X_2^N(\tau_i)\leq \cdots \leq X_N^N(\tau_i)$ for all $i$, but this ordering may not hold for $t\in (\tau_i,\tau_{i+1})$. We define $\Theta^N:\B{R}^N\to\B{R}^N$ to be the function which orders $\B{R}^N$-valued vectors, and say that $\Theta^N_i(\vec{v})$ is the $i$\textsuperscript{th} smallest element of $\vec{v}$.

Now define $b^N(t):=\inf\left\{x:\Lambda \langle \mu_t^N,\omega(\cdot-x)\rangle \geq f(t)\right\}$, where the notation $\langle \mu,g\rangle$ means $\int_\B{R} g(x)\mu(dx)$. Since $\omega$ is decreasing, $\langle \mu_t^N,1\rangle=1$, and $\Lambda > \|f\|_\infty$% , thus as $x\to\infty$, $\langle \mu_t^N,\omega(\cdot-x)\rangle \to 1$, and as $x\to -\infty$, $\langle \mu_t^N,\omega(\cdot-x)\rangle \to 0$,
, therefore $b^N(t)$ is well defined. Moreover we can observe that if $\omega$ is continuous, then $\frac{\Lambda}{N} \sum_{i=1}^N \omega(X_i^N(t)-b^N(t))$ is exactly $f(t)$. $b^N(t)$ defines a moving boundary. %If $\omega$ is discontinuous, then $\frac{\Lambda}{N}\sum_{i=1}^N \omega(X_i^N(t)-b^N(t))$ differs from $f(t)$ by at most $1/N$ almost surely. 
Then at time $\tau_1$, the particle of rank $I_1$ branches, and simultaneously particle of rank $J_1$ is deleted, where $\B{P}(I_1=i)=1/N$ and independently $\B{P}(J_1=j)=\frac{\omega(\Theta^N_j(X^N(\tau_1-))-b^N(\tau_1-))}{\sum_{\ell=1}^N \omega(\Theta^N_\ell(X^N(\tau_1-))-b^N(\tau_1-))}$%=\frac{\Lambda\omega(\Theta^N_j(X^N(\tau_1-))-b^N(\tau_1))}{Nf(t)}$ 
for $i,j\in [N]$. Repeat this inductively for each interval $[\tau_k,\tau_{k+1}]$. 

%Essentially this process deletes particles, at time $\tau_i$, according to how far to the left of the location $b^N(\tau_i)$. Viewed another way, the particle of rank $j$ is deleted at rate $\Lambda\omega(\Theta^N_j(X^N(t))-b^N(t))$ where $b^N:[0,\infty)\to \B{R}$ is a reference point chosen so that the total branching rate at time $t$ is $Nf(t)$.

\section{Main results}

The main result of this paper is the following hydrodynamic limit, which describes the limiting behaviour of the above described $(\omega,f,\Lambda,N)$-BBM process by the free boundary problem:
\begin{align} \label{omegaNBMMHydro}
    \begin{cases}
        u_t = \frac12 u_{xx} + f(t)u - \Lambda\omega(x-b(t))u & x\in \B{R},\;t>0,\\
        \int_\B{R} u(y,t)dy = 1 & t>0,\\
        u(x,0)=\rho(x) & x\in \B{R}.
    \end{cases}
    \end{align}

\begin{theorem} \label{1stMainThemOmega}
    Let $\omega:\B{R}\to[0,1]$ be a selection weighting such that either $\omega$ is continuous or $\omega(x)=\is_{x\leq 0}$. Let $(\mu_t^N)_{t\geq 0}$ be the measure-valued process describing the $(\omega,f,\Lambda,N)$-BBM process with initial condition $\rho^N$ and let $Q_T^N$ denote its law. Then for any $T\in (0,\infty)$, the sequence of measures $Q_T^N$ has a weak limit, $Q_T^\infty$. Moreover, if $(\mu_t)_{t\in [0,T]}\sim Q_T^\infty$ and $b(t):=\inf\{x:\Lambda\langle \mu_t,\omega(\cdot-x)\rangle \geq f(t)\}$, then $\mu_t$ has density $u(x,t)$ and $(u,b)$ is $Q_T^\infty$-a.s. the unique weak solution to the FBP \eqref{omegaHydroLim}.
\end{theorem}

Using the results of \cite{ettingerHeningEvans}, \cite{ettingerHeningKwong}, we also connect this to the SKIFPP. 

\begin{theorem} \label{2ndMainThemOmega}
    Suppose that $f$ is continuous, and that $\rho\in C^2(\B{R})$ with $\rho$, $\rho'$, and $\rho''$ bounded. Suppose that either $\omega(x)=\is_{x\leq 0}$ or $\omega\in C^3(\B{R})$ is non-increasing and $\{x:\omega(x)\in (0,1)\}$ is bounded. If $(u,b)$ is the unique weak solution to \eqref{omegaNBMMHydro} and $\tau$ is the stopping time:
    \begin{align*}
        \tau:=\inf\left\{t:\Lambda\int_0^t \omega(W_s-b(s))ds\geq U\right\},
    \end{align*}
    where $U\sim Exp(1)$ and $W_0\sim \rho$, then $b$ is the unique barrier such that $\B{P}(\tau>y)=\exp(-\int_0^y f(s)ds)$. That is, $b$ solves the SKIFPP. 
\end{theorem}

\begin{proof}
    %Consider first the case that $\omega(x)=\is_{x\leq 0}$. Then, defining $G(t)=\exp(-\int_0^{t/\Lambda}f(s)ds)$, by Theorem 1.8 \cite{ettingerHeningKwong}, there is a unique weak solution $(v,c)$ solving 
    Let $\tilde{\omega}(x)=\omega(\Lambda^{-1/2}x)$, and suppose that $(v,c)$ is a weak solution to the PDE:
    \begin{align}\label{pdeFromLit}
    \begin{cases}
        v_t(x,t) = \frac12 v_{xx}(x,t) -\tilde{\omega}(x-c(t))v(x,t) & x\in \B{R},\;t>0,\\
        \int_\B{R} v(y,t)dy = G(t):=\exp(-\int_0^{t/\Lambda}f(s)ds) & t>0,\\
        \int_{\B{R}} v(y,0)\omega(y-c(0))dy=f(0)\\
        v(x,0)=\Lambda^{1/2}\rho(\Lambda^{-1/2}x) & x\in \B{R}.\\
    \end{cases}
    \end{align}
    Then it is easily checked that \begin{align*}(\tilde{u}(x,t),\tilde{b}(t)):=\left(\Lambda^{1/2}\exp(\int_0^t f(s)ds)v(\Lambda^{1/2}x,\Lambda t),\Lambda^{-1/2}c(\Lambda t)\right)\end{align*}
    is a weak solution of \eqref{omegaNBMMHydro}, and therefore by assumption coincides with limit $(u,b)$ found in Theorem \ref{1stMainThemOmega}. In the case that $\omega(x)=\is_{x\leq 0}$, Theorem 1.8 \cite{ettingerHeningKwong} tells us that \eqref{pdeFromLit} has a unique weak solution and in the case that $\omega\in C^3(\B{R})$ and $\{x:\omega(x)\in (0,1)\}$ is bounded, Corollary 2.7 \cite{ettingerHeningEvans} tells us that \eqref{pdeFromLit} has a unique (classical) solution. Theorem 1.8 \cite{ettingerHeningKwong} and Corollary 2.7 \cite{ettingerHeningEvans} also tell that $c$ is the unique continuous barrier such that
    \begin{align*}
    \B{P}\left(\inf\left\{t:\int_0^t \omega(W_s-c(s))ds \geq U\right\}>y\right)=G(y)=\exp(-\int_0^{y/\Lambda}f(s)ds).
    \end{align*}
    Then using the fact that $\Lambda^{-1/2}W_{\Lambda z}\overset{d}{=}W_z$ and $c(s)=\Lambda^{1/2}b(s/\Lambda)$, this implies that
    \begin{align*}
        \B{P}\left(\inf\left\{t:\Lambda \int_0^t \omega(W_s - b(s))ds\geq U\right\}>y\right)\exp(-\int_0^y f(s)ds),
    \end{align*}
    which is to say that $b$ solves the SKIFPP. 
\end{proof}

\section{Proof of Theorem \ref{1stMainThemOmega}}

The method in this section follows similarly to Demircigil and Tomasevic \cite{demircigilTomasevic} and the author's paper \cite{me3}. As in \cite{me3}, define the set of test functions $\C{BC}^{2,1}(\B{R}\times [0,\infty),\B{R})$ to be the set of bounded and continuous functions $f(x,t):\B{R}\times [0,\infty)\to\B{R}$ with bounded and continuous derivatives $f_x$, $f_{xx}$, and $f_t$. We also use $C_b(\B{R})$ to denote the continuous bounded functions $\B{R}\to\B{R}$. For a (local) martingale $(M_t)_{t\geq 0}$, we denote the quadratic variation process by $([M]_t)_{t\geq 0}$.

\begin{propn} \label{probRepPropn}
    Let $(\mu_t^N)_{t\geq 0}$ be the measure-valued process describing the $(\omega,f,\Lambda,N)$-BBM process, and define $b^N(t):=\inf\{x:\Lambda\langle \mu_t^N,\omega(\cdot-x)\rangle\geq f(t)\}$. Then for any bounded $\omega$ and any test function $\phi(x,t)\in \C{BC}^{2,1}(\B{R}\times [0,\infty),\B{R})$ we have
    \begin{align} \label{probRepIto} \langle \mu_t^N,\phi(\cdot,t)\rangle = \langle \rho^N, \phi(\cdot,0)\rangle &+ \int_0^t \langle \mu_s^N, \frac12 \phi_{xx}(\cdot,s) + \phi_t(\cdot,s)+f(s)\phi(\cdot,s)\rangle \\
    &- \frac{f(s)\langle \mu_s^N,\phi(\cdot,s)\omega(\cdot-b^N(s))}{\langle \mu_s^N,\omega(\cdot-b^N(s))\rangle} ds
    \nonumber + M_t^{N,W} + M_t^{N,b},
    \end{align}
    where $M^{N,W}_t$ is a continuous local martingale and $M^{N,b}_t$ is a local martingale, with $\B{E}[[M^{N,W}]_t],\B{E}[[M^{N,b}]_t]\to 0$ as $N\to\infty$ for fixed $t\geq 0$. 
\end{propn}

\begin{proof}
    Let the branching times of the particle system be $\tau_1<\tau_2<\cdots$. Then by Ito's formula, for $t\in (\tau_{k-1},\tau_k)$, we have:
    $$\phi(X_i^N(t),t)=\phi(X_i^N(\tau_{k-1}),\tau_{k-1})+\int_{\tau_{k-1}}^t \phi_t(X_i^N(s),s)+\frac12 \phi_{xx}(X_i^N(s),s)\rangle ds + \int_{\tau_{k-1}}^t \phi_x(X_i^N(s),s)dW_i(s),$$
    so we have that:
    $$\langle \mu_t^N,\phi(\cdot,t)\rangle = \langle \mu_{\tau_{k-1}},\phi(\cdot,\tau_{k-1})\rangle + \int_{\tau_{k-1}}^t \langle \mu_s^N, \phi_t(\cdot,s)+\frac12 \phi_{xx}(\cdot,s)ds + \sum_{i=1}^N\int_{\tau_{k-1}}^t \phi_x(X_i^N(s),s) dW_i(s).$$
    Then, at time $\tau_k$, the quantity $\langle \mu_t^N, \phi(\cdot,t)\rangle$ increases by $\frac1N(\phi(\Theta^N_{I_i}(X^N(\tau_k-)),\tau_k)-\phi(\Theta^N_{J_i}(X^N(\tau_k-)),\tau_k))$ due to branching and deletion. Therefore integrating over $[0,t]$, combining the terms due to Brownian motion and the terms due to branching and selection, we yield:
    \begin{align*}
        \langle \mu_t^N,\phi(\cdot,t)\rangle = \langle \rho^N,\phi(\cdot,0)\rangle &+ \int_0^t \langle \mu_s^N,\frac12 \phi_{xx}(\cdot,s)+\phi_t(\cdot,s)\rangle + M^{N,W}_t \\
        &+ \frac1N\int_0^t \phi(\Theta^N_{I_{\C{N}(s)}}(X^N(s-)),s)-\phi(\Theta^N_{J_{\C{N}(s)}}(X^N(s-)),s)\C{N}(ds)
    \end{align*}
    where $M^{N,W}_t:=\frac1N \sum_{i=1}^N \int_0^t \phi_x(X_i^N(s),s)dW_i(s)$ is a continuous local martingale. Next we prove that
    \begin{align*}
        M^{N,b}_t:=\frac1N \int_0^t \phi(\Theta^N_{I_{\C{N}(s)}}&(X^N(s-)),s)-\phi(\Theta^N_{J_{\C{N}(s)}}(X^N(s-)),s)\C{N}(ds) \\
        &- \int_0^t f(s)\langle \mu_s^N,\phi(\cdot,s)\rangle - \frac{f(s)\langle \mu_s^N,\phi(\cdot,s)\omega(\cdot-b^N(s))\rangle}{\langle \mu_s^N,\omega(\cdot-b^N(s))\rangle}ds
    \end{align*}
    is a local martingale. Certainly if $\C{F}_t:=\sigma(W_i(s),\C{N}(s),I_j,J_j:s\leq t, i\in [N],j\leq \C{N}(t))$ is the natural filtration of the process, then $\B{E}[M^{N,b}_t-M^{N,b}_s|\C{F}_s]=0$. This follows by the independence of $I$ and $J$ from $\C{N}$ so that:
    \begin{align*}
        \B{E}\Bigg[&\frac1N\int_s^t
        \phi(\Theta^N_{I_{\C{N}(u)}}(X^N(u-)),u)-\phi(\Theta^N_{J_{\C{N}(u)}}(X^N(u-)),u)\C{N}(du)|\C{F}_s\Big]\\
        &=\B{E}\Bigg[\frac1N\int_s^t \phi(\Theta^N_{I_{\C{N}(u)}}(X^N(u-)),u)-\phi(\Theta^N_{J_{\C{N}(u)}}(X^N(u-)),u)\C{N}(du)\Bigg]\\
        &=\B{E}\Bigg[\frac1N\sum_{i,j\in[N]}\int_s^t \big(\phi(\Theta^N_i(X^N(u-)),u)-\phi(\Theta^N_j(X^N(u-)),u)\big)\B{P}(I_{\C{N}(u)}=i,J_{\C{N}(u)}=j)\C{N}(du)]\\
        &=\B{E}\Bigg[\frac{1}{N^2}\sum_{i,j\in[N]}\int_s^t \big( \phi(\Theta^N_i(X^N(u-)),u)-\phi(\Theta^N_j(X^N(u-)),u)\big)\frac{\omega(\Theta^N_j(X^N(u-))-b^N(u-))}{\sum_{k=1}^N\omega(\Theta^N_k(X^N(u-))-b^N(u-))}\C{N}(du)\Bigg]\\
        &=\B{E}\Bigg[\int_s^t f(u)\langle \mu_u^N,\phi(\cdot,u)\rangle - \frac{f(u)\langle \mu_u^N,\phi(\cdot,u)\omega(\cdot-b^N(u))\rangle}{\langle \mu_u^N,\omega(\cdot-b^N(u))\rangle}du\Bigg],
    \end{align*}
    where the final equality follows from the fact that $\C{N}(t)-N\int_0^t f(s)ds$ is a martingale (the `compensated Poisson process'). Therefore
    %Since the Poisson process $\C{N}(t)$ has intensity $Nf(t)$, we know that the process $\C{N}(t)-Nf(t)$ is a martingale (known as the compensated Poisson process) and furthermore (by the continuity of $f$), the process 
    %$$\frac{\omega(X_j^N(t-)-b^N(t-))(\phi(X_j^N(t-),t)-\phi(X_i^N(t-),t))}{\sum_{\ell=1}^N \omega(X_\ell^N(t-)-b^N(t-))}$$ 
    %is predictable, therefore:
    %$$M^{N,b}_t := \int_0^t \sum_{1\leq i,j\leq N}\frac{1}{N^2}\frac{\omega(X_j^N(t-)-b^N(t-))(\phi(X_j^N(t-),t)-\phi(X_i^N(t-),t))}{\sum_{\ell=1}^N \omega(X_\ell^N(t-)-b^N(t-))}(\C{N}(dt)-Nf(t)dt)$$ is a martingale. Thus we can write
    \begin{align*}
        \langle \mu_t^N,\phi(\cdot,t)\rangle = \langle \rho^N,&\phi(\cdot,0)\rangle + \int_0^t \langle \mu_s^N,\frac12 \phi_{xx}(\cdot,s)+\phi_t(\cdot,s)\rangle ds + M^{N,W}_t + M^{N,b}_t \\
        &+\int_0^t f(s)\langle \mu_s^N,\phi(\cdot,s)\rangle - \frac{f(s)\langle \mu_s^N,\phi(\cdot,s)\omega(\cdot-b^N(s))\rangle}{\langle \mu_s^N,\omega(\cdot-b^N(s))\rangle} ds 
        %&\int_0^t \sum_{1\leq i,j\leq N}\frac{\Lambda}{N^2} (\phi(X_i^N(s-),s)-\phi(X_j^N(s-),s))\omega(X_j^N(s-)-b^N(s-))ds %\\
        %= \langle \mu_0^N,\phi(\cdot,0)\rangle &+ \int_0^t \langle \mu_s^N,\frac12 \phi_{xx}(\cdot,s)+\phi_t(\cdot,s)\rangle ds + M^{N,W}_t + M^{N,b}_t \\
        %&+ \int_0^t \langle \mu_s^N,\phi(\cdot,s)\rangle \langle \mu_s^N,\omega(\cdot-b^N(s))\rangle - \langle \mu_s^N,\phi(\cdot,s)\omega(\cdot-b^N(s))\rangle ds \\
        %= \langle \mu_0^N,\phi(\cdot,0)\rangle &+ \int_0^t \mu_s^N,\frac12 \phi_{xx}(\cdot,s)+\phi_t(\cdot,s)\rangle +  \langle \mu_s^N,\phi(\cdot,s)\rangle f(s) - \langle \mu_s^N,\phi(\cdot,s)\omega(\cdot-b^N(s))\rangle ds\\
        %&+ M^{N,W}_t + M^{N,b}_t.
    \end{align*}
    It just remains to prove that $\B{E}[[M^{N,W}]_t]$ and $\B{E}[[M^{N,b}]_t]$ converge to $0$ as $N\to\infty$. For fixed $t\geq 0$, It\^{o}'s isometry gives:
    $$\B{E}\left[ \frac1N \sum_{i=1}^N \int_0^t \phi(X_i^N(s),s)dW_i(s)\right]_t = \frac{1}{N^2}\sum_{i=1}^N \int_0^t \phi_x(X_i^N(s),s)^2 d[W_i]_s = \frac1N \int_0^t \langle \mu_s^N, \phi_x(\cdot,s)^2\rangle ds.$$
    Thus since $\phi_x$ is bounded, this quantity converges to $0$ as $N\to\infty$. We also calculate that:
    \begin{align*}
        \B{E}&[[M^{N,b}]_t] \\
        &=   \B{E}\left[\int_0^t \frac{1}{N^2}\sum_{i,j\in[N]}\frac{\omega(X_j^N(s-)-b^N(s-))(\phi(X_j^N(s-),s)-\phi(X_i^N(s-),s))}{\sum_{k=1}^N \omega(\Theta^N_k(X^N(s-))-b^N(s-))}(\C{N}(ds)-Nf(s)ds)\right]_t \\
        &= \int_0^t \left(\frac{1}{N^2}\sum_{i,j\in[N]}\frac{\omega(X_j^N(s-)-b^N(s-))(\phi(X_j^N(s-),s)-\phi(X_i^N(s-),s))}{\sum_{k=1}^N\omega(\Theta^N_k(X^N(s-))-b^N(s-))}\right)^2 d[\C{N}(\cdot)-Nf(\cdot)]_s
    \end{align*}
    Now $d[\C{N}(\cdot)-Nf(\cdot)]_s=Nf(s)ds$, and since $\phi$ is bounded, say by $\Phi$, we have 
    \begin{align*}
        \Bigg(\frac{1}{N^2}\sum_{i,j\in[N]}&\frac{\omega(X_j^N(s-)-b^N(s-))(\phi(X_j^N(s-),s)-\phi(X_i^N(s-),s))}{\sum_{k=1}^N\omega(\Theta^N_k(X^N(s-))-b^N(s-))}\Bigg)^2 \\
        &\leq \left( \sum_{i \in [N]} \frac{2\Phi}{N^2}\frac{\sum_{j=1}^N\omega(X_j^N(s-)-b^N(s-))}{\sum_{k=1}^N \omega(X_k^N(s-)-b^N(s-))}\right)^2 = \left(\sum_{i=1}^N \frac{2\Phi}{N^2}\right)^2 = \frac{4\Phi^2}{N^2},
    \end{align*}
    therefore, since $f$ is bounded by $\Lambda$,
    $$\B{E}[M^{N,b}]_t \leq \int_0^t \frac{4\Phi^2}{N}f(s)ds \leq \frac{4\Lambda\Phi^2 t}{N} \xrightarrow[N\to\infty]{} 0,$$
    thus concluding the proof. 
\end{proof}

Since we seek the limit of the sequence $((\mu_t^N)_{t\geq 0})_{N=1,2,\ldots}$, we then need to show that the sequence of laws is tight, which is the substance of the next theorem. This proposition follows the structure of Proposition 5, \cite{me3}. In particular, we show tightness in the space $\C{P}(\C{D}([0,T],\C{M}_F^w(\bar{\B{R}})))$, where $\C{M}_F^w(\bar{\B{R}})$ is the space of finite measures on the extended real line under the weak topology. We do this as it will help to prove a compact containment condition. %After proving tightness, we will subsequently prove that any limit of the sequence, $\mu_t^\infty$, has no mass on $\bar{\B{R}}\setminus \B{R}

\begin{propn} \label{tightness}
    Fix $T>0$, and let $Q^N_T$ denote the law of $(\mu_t^N)_{t\in [0,T]}$. Then the sequence $(Q^N_T)_{N=1,2,\ldots}$ is tight in $\C{P}(\C{D}([0,T],\C{M}^w_F(\bar{\B{R}})))$ with respect to the Skorokhod topology. 
\end{propn}

\begin{proof}
    By Theorem 1.18, \cite{etheridge}, in order to show tightness, it is sufficient to show the following conditions:
    \begin{enumerate}[noitemsep]
        \item[(i)] \textit{Compact containment:} For all $\epsilon > 0$ there exists a compact subset $K_\epsilon\subseteq \C{M}_F^w(\bar{\B{R}})$ such that $$\inf_{N\in \B{N}} \B{P}(\mu_t^N \in K_\epsilon\;  \forall \, t\in [0,T])>1-\epsilon.$$
        \item[(ii)] \textit{Tightness of real-valued processes:} The sequence of laws of $(\langle \mu_t^N,f\rangle)_{0\leq t\leq T}$ is tight in $\C{P}(\C{D}([0,T],\B{R}))$ for every function $f$ in a dense subset of $C_b(\bar{\B{R}})$.
    \end{enumerate}
    Since the set $\{\mu:\mu(\bar{\B{R}})=1\}$ is a compact subset of $\C{M}_F^w(\bar{\B{R}})$ and $\mu_t^N(\bar{\B{R}})=1$ for all $t\geq 0$ and $N\in \B{N}$, thus condition (i) holds immediately. Next we show condition (ii). Since each $\mu_t^N$ is supported on $\B{R}$, condition (ii) is equivalent to showing tightness of the sequence of laws of $(\langle \mu_t^N,f\rangle )_{t\in [0,T]}$ for all $f$ in a dense subset of $C_b(\B{R})$; we will take $\C{BC}^{2,1}(\B{R}\times [0,\infty),\B{R})\cap C_b(\B{R})$ for our dense subset. %, since By Theorem 2.1 of \cite{roellyCop}, all that is required to show tightness in $\C{D}([0,T],\C{M}_F(\B{R}))$ is to show that for every $\phi$ in a dense subset of $C_0(\B{R})$, $((\langle \mu_t^N, \phi\rangle)_{t\in [0,T]})_{N=1,2,\ldots}$ is tight in $\C{D}([0,T],\B{R})$. Note that $\C{BC}^{2,1}(\B{R}\times [0,\infty),\B{R})\cap C_0(\B{R})$ is a dense subset of $C_0(\B{R})$, so fix $\phi\in \C{BC}^{2,1}(\B{R}\times [0,\infty),\B{R})$ which vanishes at $\pm \infty$. Note that this is a dense subset of $C_0(\B{R})$. 
    By Aldous' tightness criterion (see, for example, Theorem 16.10, \cite{billingsley}), the sequence of laws of $((\langle \mu_t^N,\phi\rangle)_{t\in [0,T]})_{N=1,2,\ldots}$ is tight if:
    \begin{enumerate}[noitemsep]
        \item[A] For every $m>0$, $\lim_{a\to\infty}\limsup_{N\to\infty}\B{P}(\sup_{0\leq t\leq m}|\langle \mu_t^N,\phi\rangle|\geq a)=0.$
        \item[B] For each $\epsilon, \eta, m>0$, there exists $\delta_0,N_0$ such that if $\delta\leq \delta_0$ and $N\geq N_0$ and $\tau$ is a stopping time $\tau\leq m$ then $\B{P}(|\langle \mu_{\tau+\delta}^N,\phi\rangle - \langle \mu_\tau^N,\phi\rangle |\geq \epsilon)\leq \eta$. 
    \end{enumerate}
    Since $|\langle \mu_t^N,\phi\rangle|\leq \|\phi\|_\infty$, condition A follows immediately. For condition B, note that by \eqref{probRepIto} and the boundedness of $f$, $\phi$, $\phi_t$, $\phi_{xx}$, $\omega$, and $\Lambda$, there exists constant $\kappa$ such that 
    \begin{align*}
        |\langle \mu_{\tau+\delta}^N,\phi\rangle - \langle \mu_\tau^N,\phi\rangle |\leq \kappa\delta + |M^{N,W}_{\tau+\delta}-M^{N,W}_{\tau}|+|M^{N,b}_{\tau+\delta}-M^{N,b}_{\tau}|,
    \end{align*}
    therefore by the Markov inequality and the Burkholder-Davis-Gundy inequality, there exists a constant $\hat{\kappa}$
    \begin{align*}
        \B{P}(|\langle \mu_{\tau+\delta}^N,\phi,\rangle-\langle \mu_\tau^N,\phi\rangle|\geq \epsilon)&\leq \is_{\kappa\delta\geq \epsilon/3}+\B{P}(|M_{\tau+\delta}^{N,W}-M^{N,W}_\tau|^2\geq \epsilon^2/9)+\B{P}(|M_{\tau+\delta}^{N,b}-M_{\tau}^{N,b}|^2\geq \epsilon^2/9)\\
        &\leq \is_{\kappa\delta\geq \epsilon/3}+\frac{9}{\epsilon^2}\B{E}[|M_{\tau+\delta}^{N,W}-M^{N,W}_\tau|^2]+\frac{9}{\epsilon^2}\B{E}[|M_{\tau+\delta}^{N,b}-M^{N,b}_\tau|^2]\\
        &\leq \is_{\kappa\delta \geq \epsilon/3}+\frac{9\hat{\kappa}}{\epsilon^2}(\B{E}[[M^{N,W}]_\delta]+\B{E}[[M^{N,b}]_\delta])\xrightarrow[\delta\to 0,N\to\infty]{}0.
    \end{align*}
    This shows condition B, and therefore $((\langle \mu_t^N,\phi\rangle)_{t\in [0,T]})_{N=1,2,\ldots}$ is tight in $\C{D}([0,T],\B{R})$. Hence condition (ii) holds, and thus the sequence $(Q^N_T)_{N=1,2,\ldots}$ is tight in $\C{D}([0,T],\C{M}^w_F(\bar{\B{R}}))$.
\end{proof}

We need not worry about the generalisation from $\C{M}_F^w(\B{R})$ to $\C{M}_F^w(\bar{\B{R}})$ since the next theorem shows that if $Q^\infty_T$ is a subsequential limit of $(Q^N_T)_{N\geq 1}$ then $Q^\infty_T$-almost all measure-valued processes have no mass on $\{\pm\infty\}$ for any $t\in [0,T]$. That is, if $(\mu_t)_{t\in [0,T]}\sim Q^\infty_T$, then $\mu:[0,T]\mapsto \C{M}_F^w(\B{R})$ almost surely. It also shows that $\mu_t$ has a density for all $t\in [0,T]$ $Q^\infty_T$-almost surely.

\begin{propn}
    Let $Q^\infty_T$ be a subsequential limit of $(Q^N_T)_{N\geq 1}$ and let $(\mu_t)_{t\in [0,T]}\sim Q^\infty_T$. Then $\mu_t\ll\lambda$ and $\mu_t(\{-\infty,\infty\})=0$ for all $t\in [0,T]$ $Q_T^\infty$-a.s. 
\end{propn}

\begin{proof}
    Consider a branching Brownian motion (BBM) $(X_u(t):u\in\C{U}(t))_{t\geq 0}$ starting from initial configuration $\rho^N$ and branching at the time-dependent rate $f$. Let $\C{U}(t)$ denote the set of particles at time $t$ and $\C{U}^i(t)$ denote the subset of $\C{U}(t)$ consisting of the descendants of the particle which starts at the location of the $i$\textsuperscript{th} leftmost particle of $\rho^N$. We will colour the particles of the BBM blue and red, and at each branching event, a particle will branch into two particles of like colour. 
    Initially colour all particles blue. Then subsequently, at each branching time $\tau_k$ of a blue particle, colour exactly one blue particle red, choosing the $j$\textsuperscript{th} leftmost blue particle with probability $\frac{\omega(\Theta^N_j(X^N(\tau_k-))-b^N(\tau_k-))}{\sum_{k=1}^N\omega(\Theta^N_k(X^N(\tau_k-))-b^N(\tau_k-))}$. By construction, the subset of blue particles describes a $(\omega,f,\Lambda,N)$-BBM process. Let $\hat{\mu}^N(t)=\frac1N\sum_{u\in\C{U}(t)}\delta_{X_u(t)}$ be the rescaled empirical measure of this coloured BBM at time $t$. By the above coupling, it is clear that $\mu_t^N(A)\leq \hat{\mu}_t^N(A)$ for any measurable set $A$. It is well known (see Appendix of \cite{me3} for a proof of the exact statement) that if $\hat{Q}^N_T$ is the law of $(\hat{\mu}^N_t)_{t\in [0,T]}$, then $\hat{Q}^N_T$ converges weakly to $\hat{Q}^\infty_T$, and if $(\hat{\mu}_t)_{t\in [0,T]}\sim \hat{Q}^\infty_T$, then $\hat{\mu}_t(dx)=u(x,t)dx$ $\hat{Q}^\infty_T$-a.s., where $u(x,t)=\exp(\int_0^t f(s)ds)\frac{\partial}{\partial x}\B{P}_\rho(B(t)\leq x)$ for Brownian motion $B$ with initial distribution $\rho$. In particular, $u$ is the unique classical solution to the PDE $u_t = \frac12 u_{xx}+f(t)u$. Then since $\hat{\mu}_t\ll \lambda$ and $\hat{\mu}_t(\{-\infty,\infty\})=0$, and $\hat{\mu}^N_t$ dominates $\mu_t^N$ for all $N$, thus $\mu_t\ll \lambda$ and $\mu_t(\{-\infty,\infty\})=0$ for all $t\in [0,T]$. 
\end{proof}

\begin{propn} \label{vagueContinuity}
    Let $Q^\infty_T$ be a subsequential limit of $(Q^N_T)_{N=1,2,\ldots}$ and let $(\mu_t)_{t\in [0,T]}\sim Q^\infty_T$. Then the map $t\mapsto \mu_t$ is continuous on $[0,T]$ with respect to the weak topology $Q^\infty_T$-a.s.
\end{propn}

\begin{proof}
    Recall the Kolmogorov continuity theorem, which states that if $X:[0,\infty)\times \Omega\to S$ is a stochastic process and $(S,d)$ a metric space, and there exist positive $\alpha,\beta,K$ such that $\B{E}[d(X_s,X_t)^\alpha]\leq K|t-s|^{1+\beta}$ for all $s,t$, then there exists a modification of $X$ which is continuous. So fix $s,t$ and $\phi\in \C{BC}^{2,1}(\B{R}\times [0,\infty),\B{R})$. Since $Q^N_T$ converges to $Q_T^\infty$ in the Skorokhod topology, there exists a sequence $\ell_N:[0,T]\to [0,T]$ such that $\sup_{t\in [0,T]}|\ell_N(t)-t|$ and $\sup_{t\in [0,T]}|\langle \phi,\lambda^N_{\ell_N(t)}\rangle-\langle \phi,\mu_t\rangle|$ converge to $0$ as $N\to\infty$. Then by Fatou's lemma:
    $$\B{E}[|\langle \phi,\mu_s\rangle - \langle \phi,\mu_t\rangle|^2]=\B{E}[\lim_{N\to\infty}|\langle \phi,\mu^N_{\ell_N(s)}\rangle - \langle \phi,\mu^N_{\ell_N(t)}\rangle|^2]\leq \liminf_{N\to\infty}\B{E}[|\langle \phi,\mu_{\ell_N(s)}^N\rangle - \langle \phi,\mu^N_{\ell_N(t)}\rangle |^2].$$
    Then as $\omega$, $\Lambda$, $f$, $\phi$, $\phi_t$, and $\phi_{xx}$ are bounded, and $\langle\mu_s^N,1\rangle\equiv 1$, thus \eqref{probRepIto} yields that there exists a constant $\kappa$ such that
    \begin{align*}
        \B{E}[|\langle \phi,\mu_{\ell_N(s)}^N\rangle - \langle \phi, \mu_{\ell_N(t)}^N\rangle|^2]\leq 9\left(\kappa|\ell_N(s)-\ell_N(t)|^2+\B{E}[|M^{N,W}_{\ell_N(s)}-M^{N,W}_{\ell_N(t)}|^2]+\B{E}[|M^{N,b}_{\ell_N(s)}-M^{N,b}_{\ell_N(t)}|^2]\right).
    \end{align*}
    By the Bukrholder-Davis-Gundy inequality and Proposition \ref{probRepPropn}, $\B{E}[|M^{N,W}_{\ell_N(s)}-M^{N,W}_{\ell_N(t)}|^2]$ and $\B{E}[|M^{N,b}_{\ell_N(s)}-M^{N,b}_{\ell_N(t)}|^2]$ converge to $0$ as $N\to\infty$. Moreover, as $\sup_{t\in [0,T]}|\ell_N(t)-t|\xrightarrow[N\to\infty]{}0$, thus
    \begin{align*}\liminf_{N\to\infty}\B{E}[|\langle \phi,\mu_{\ell_N(s)}^N\rangle - \langle \phi, \mu_{\ell_N(t)}^N\rangle|^2]\leq 9\kappa|s-t|^2.\end{align*}
    Therefore by the Kolmogorov continuity theorem, $t\mapsto\langle \phi,\mu_s\rangle$ has a continuous modification on $[0,T]$, and thus, since it is a cadlag process, is almost surely continuous on $[0,T]$ (see Theorem 1, \cite{schilling}). Therefore $t\mapsto\langle \phi,\mu_t\rangle$ is continuous for all $\phi$ in a countable dense subset of $C_b(\B{R})$ and therefore $(\mu_s)_{t\in [0,T]}$ is almost surely continuous with respect to the weak topology. 
\end{proof}

\begin{lemma} \label{bConvLemma}
    Suppose that the measure $\mu^N_s$ converges weakly to the measure $\mu_s\ll\lambda$ almost surely and $b(s):=\inf\{x:\Lambda\langle \mu_s,\omega(\cdot-x)\rangle \geq f(s)\}$. If either $\omega$ is continuous or $\omega = \is_{\{x\leq 0\}}$, then for any $\phi\in C_b(\B{R})$, %$\langle \mu_s^N,\omega(\cdot-b^N(s))\rangle\to \langle \mu_s,\omega(\cdot-b(s))\rangle$ and 
    $\langle \mu^N_s,\phi(\cdot)\omega(\cdot-b^N(s))\rangle\to\langle \mu_s, \phi(\cdot)\omega(\cdot-b(s))\rangle$ almost surely. 
\end{lemma}

\begin{proof}
    Fix an element $\sigma\in \Omega$ in our state space at which $\mu_s^N(\sigma)$ converges weakly to $\mu_s(\sigma)$. We will omit dependence on $\sigma$ here onwards for readability. 
    
    Consider the first case that $\omega$ is continuous. %In this case $\langle \mu_s^N,\omega(\cdot-b^N(s))\rangle = \Lambda f(s) = \langle \mu_s,\omega(\cdot-b(s))\rangle$. 
    Define the function $B(x,\nu):=\langle \omega(\cdot-x),\nu\rangle$. Then for any probability measure on $\B{R}$, $B$ is a continuous and non-decreasing function of $x$ with $\lim_{x\to\infty}B(x,\nu)=1$, $\lim_{x\to-\infty}B(x,\nu)=0$, and $B(b^N(s),\mu^N_s)=f(s)/\Lambda=B(b(s),\mu_s)$ for any $s$. %This proves the first convergence. 
    Since $\omega$ is continuous and bounded, thus by weak convergence, $B(x,\mu^N_s)\to B(x,\mu_s)$. Moreover, since $B(x,\mu^N_s)$ is monotonic increasing for each $N$ and $s$ and converges pointwise to $B(x,\mu_s)$, thus convergence is uniform in $x$ (see \cite{chowTeicher}, \S 8.2, Lemma 3, for example). Now let $b^\star(s)=\lim_{k\to\infty}b^{N_k}(s)$ be a subsequential limit. Therefore $B(b^{N_i}(s),\mu^N_s)$ converges to $B(b^{N_i}(s),\mu_s)$ uniformly for all $i$. Finally, as $B(x,\nu)$ is continuous in $x$, we can conclude that 
    \begin{align*}
        B(b(s),\mu)= f(s)/\Lambda = \lim_{k\to\infty}B(b^{N_k}(s),\mu^{N_k})=\lim_{k\to\infty}B(b^{N_k}(s),\mu)=B(\lim_{k\to\infty}b^{N_k}(s),\mu)=B(b^\star(s),\mu)
    \end{align*}
    Therefore $b^\star(s)\geq b(s)$. Then as $x\mapsto \omega(y-x)$ is non-decreasing, $\omega(y-b^\star(s))-\omega(y-b(s))$ is non-negative for all $y$. Then $0=f(s)/\Lambda- f(s)/\Lambda=B(b^\star(s),\mu)-B(b(s),\mu)=\int_{\B{R}}\omega(y-b^\star(s))-\omega(y-b(s))\mu(dy)$, which is to say that $\omega(y-b^\star(s))=\omega(y-b(s))$ for $\mu$-almost all $y$. By the triangle inequality:
    \begin{align*}
        |\langle &\mu_s^{N_k},\phi(\cdot)\omega(\cdot-b^{N_k}(s))\rangle - \langle \mu_s,\phi(\cdot)\omega(\cdot-b^\star(s))\rangle| \\
        &\leq |\langle \mu_s^{N_k},\phi(\cdot)\omega(\cdot-b^{N_k}(s))\rangle - \langle \mu_s,\phi(\cdot)\omega(\cdot-b^{N_k}(s))\rangle| + |\langle \mu_s,\phi(\cdot)\omega(\cdot-b^{N_k}(s))\rangle - \langle \mu_s,\phi(\cdot)\omega(\cdot-b^\star(s))\rangle|.
    \end{align*}
    The first term on the right hand side of the inequality converges to $0$ by weak convergence, since $x\mapsto \phi(x)\omega(x-b^N(s))$ is bounded and continuous. Moreover, since $\omega$ is bounded, continuous, and monotonic, thus $\omega$ is uniform continuous, and therefore $\omega(x-b^{N_k})$ converges to $\omega(x-b^\star)$ uniformly for $x\in \B{R}$. Therefore the second term is converges to $0$ as $k\to\infty$. Finally, we note that, as $\omega(y-b^\star(s))=\omega(y-b(s))$ $\mu$-a.s., thus $\langle \mu_s,\phi(s)\omega(\cdot-b^\star(s))\rangle = \langle \mu_s,\phi(s)\omega(\cdot-b(s))\rangle$. 

    Next consider the case $\omega=\is_{x\leq 0}$, so $b(s)=\inf\{x:\Lambda\mu_s((-\infty,x])\geq f(s)\}$ and $\omega(x-b(s))=\is_{\{\Lambda\mu_s((-\infty,x])<f(s)\}}$. Similarly $b^N(s)=\inf\{x:\Lambda\mu_s^N((-\infty,x])\geq f(s)\}$ and $\omega(x-b^N(s))=\is_{\{\Lambda\mu_s^N((-\infty,x])<f(s)\}}$. 
    
    Since, $\mu_s\ll\lambda$, so $\mu_s((-\infty,x])$ is continuous, thus again, as $\mu_s^N((-\infty,x])$ is monotonic, $\mu^N_s((-\infty,x])$ converges to $\mu_s((-\infty,x])$ uniformly in $x$ (see \cite{chowTeicher}, \S 8.2, Lemma 3, for example). Moreover, for any continuous and bounded function $\phi$, $\int_{-\infty}^z \phi(x)\mu_s^{N_k}(dx)$ is the distribution function of a rescaled probability measure (which we will call $\phi\mu_s^{N_k}$) and as an immediate consequence of the weak convergence $\mu_s^{N_k}\Rightarrow \mu_s$, we have $\phi\mu_s^{N_k}\Rightarrow \phi\mu_s$. Therefore we also have that $\int_{-\infty}^z \phi(x)\mu_s^{N_k}(dx)$ converges to $\int_{-\infty}^z \phi(x)\mu_s(dx)$ uniformly in $z$. 
    
    Now suppose that $b^\star(s)=\lim_{k\to\infty}b^{N_k}(s)$ is a subsequential limit. Therefore $\mu^{N_k}_s((-\infty,b^{N_k}(s)])\to\mu_s((-\infty,b^\star(s)])$. Moreover, for any $N$ and $\epsilon>0$, $\Lambda\mu^N_s((-\infty,b^N(s)])\geq f(s)$ and $\Lambda\mu^N_s((-\infty,b^N(s)-\epsilon])<f(s)$. Therefore $\Lambda\mu_s((-\infty,b^\star(s)])\geq f(s)$ and $\Lambda\mu_s((-\infty,b^\star(s)-\epsilon])\leq f(s)$ for any $\epsilon$. Thus $\Lambda\mu_s((-\infty,b^\star(s)])=f(s)$. Therefore if $b^\star(s)>b(s)$, then $\mu_s((b(s),b^\star(s)])=0$. Then
    \begin{align}\nonumber
        |\langle \mu_s^{N_k},\phi(\cdot)\omega(\cdot-b^{N_k}(s))\rangle - \langle \mu_s,\phi(\cdot)\omega(\cdot-b(s))\rangle| = \Big|\int_{-\infty}^{b^{N_k}(s)}\phi(x)\mu^{N_k}_s(dx)-\int_{-\infty}^{b(s)}\phi(x)\mu_s(dx)\Big|\\
        \nonumber\leq \Big|\int_{-\infty}^{b^{N_k}(s)}\phi(x)\mu^{N_k}_s(dx)-\int_{-\infty}^{b^{N_k}(s)}\phi(x)\mu_s(dx)\Big| + \Big|\int_{-\infty}^{b^{N_k}(s)}\phi(x)\mu_s(dx)-\int_{-\infty}^{b^\star(s)}\phi(x)\mu_s(dx)\Big|\\
        +\Big|\int_{-\infty}^{b^\star(s)}\phi(x)\mu_s(dx)-\int_{-\infty}^{b(s)}\phi(x)\mu_s(dx)\Big|\\
        \label{triangleBConv}\leq \sup_{z\in \B{R}}\Big|\int_{-\infty}^z \phi(x)\mu_s^{N_k}(dx)-\int_{-\infty}^z \phi(x)\mu_s(dx)\Big| + \|\phi\|_\infty (\mu_s((b^{N_k}(s),b^\star(s)])+\mu_s((b(s),b^\star(s)])).
    \end{align}
    By uniform convergence in $z$, the first term converges to $0$; by the facts that $b^{N_k}(s)\to b^\star(s)$ and $\mu_s\ll\lambda$, the second term converges to $0$; and the third term is $0$. This proves the theorem. 
\end{proof}

\begin{propn} \label{weakRepPropn}
    Let $Q^\infty_T$ be a subsequential limit of $(Q^N_T)_{N=1,2,\ldots}$ and let $(\mu_t)_{t\in[0,T]}\sim Q^\infty_T$. %and define $b(t):=\inf\{x:\Lambda \langle \mu_t,\omega(\cdot-x)\rangle\geq f(t)\}$. 
    Suppose that $w$ is either continuous or $w(x)=\is_{\{x\leq 0\}}$. Then for any $\phi(x,t)\in \C{BC}^{2,1}(\B{R}\times [0,\infty),\B{R})$ we have:
    \begin{align}\label{convOfWeakRep}\langle \mu_t,\phi(\cdot,t)\rangle = \langle \rho, \phi(\cdot,0)\rangle + \int_0^t \langle \mu_s, \frac12 \phi_{xx}(\cdot,s)+\phi_t(\cdot,s)+f(s)\phi(\cdot,s)-\Lambda\phi(\cdot,s)\omega(\cdot-b^\mu(s))\rangle ds,\end{align}
    $Q^\infty_T$-almost surely.%, so that the density $u(x,t)$ of $\mu_t$ is a weak solution to the PDE:
    %$$u_t = \frac12 u_{xx} + f(s)u(x,t)-\Lambda\omega(x-b(t))u(x,t),$$
    %with the free boundary determined by the integral condition $\int_\B{R}u(x,t)dx=1$.
\end{propn}

\begin{proof}
    Without loss of generality, we will consider the convergent subsequence to be $(Q^N_T)_{N=1,2,\ldots}$ to avoid using sub-subscript. For measure-valued process $\mu=(\mu_t)_{t\geq 0}$, define $b^\mu(t):=\inf\{x:\Lambda\langle\mu_s,\omega(\cdot-x)\rangle\geq f(t)\}$ and note that $b^N=b^{\mu^N}$. Then define the function:
    \begin{align*} G(\mu,\phi,t):=\langle \mu_t,\phi(\cdot,t)\rangle - \langle \rho, \phi(\cdot,0)\rangle - \int_0^t \langle \mu_s, \frac12 \phi_{xx}(\cdot,s)+\phi_t(\cdot,s)+f(s)\phi(\cdot,s)\rangle-\\\frac{f(s)\langle\mu_s,\phi(\cdot,s)\omega(\cdot-b^\mu(s))\rangle}{\langle \mu_s,\omega(\cdot-b^\mu(s))\rangle} ds.\end{align*}
    By Proposition \ref{probRepIto}, $G(\mu^N,\phi,t)=M^{N,W}_t+M^{N,b}_t$, and so by the Burkholder-Davis-Gundy inequality, since $\B{E}[[M^{N,W}]_t]$ and $\B{E}[[M^{N,b}]_t]$ both converge to $0$ as $N\to\infty$, there exists a constant $\kappa$ such that
    $$\B{E}[G(\mu^N,\phi,t)^2]=\B{E}[(M_t^{N,W}+M_t^{N,b})^2]\leq 4\B{E}[(M_t^{N,W})^2+(M_t^{N,b})^2]\leq 4\kappa(\B{E}[[M^{N,W}]_t]+\B{E}[[M^{N,b}]_t])\xrightarrow[N\to\infty]{}0.$$
    Now by the Skorokhod representation theorem, there exists a sequence of random measure-valued processes $(\nu_t^N)_{t\in [0,T]}$ and $(\nu_t)_{t\in [0,T]}$ defined on the same probability space $(\Omega,\C{F},\B{P})$ such that $(\nu_t^N)_{t\in [0,T]}$ converges in the Skorokhod topology to $(\nu_t)_{t\in [0,T]}$ $\B{P}$-almost surely, with $(\nu^N_t)_{t\in [0,T]}\sim Q_T^N$ and $(\nu_t)_{t\in [0,T]}\sim Q_T^\infty$. By the definition of convergence in the Skorokhod topology on $\C{D}([0,T],\C{M}^w_F(\B{R}))$, there exists a sequence $\ell_N:[0,T]\to[0,T]$ such that $\sup_{t\in [0,T]}|\ell_N(t)-t|\xrightarrow[N\to\infty]{}0$ and $\sup_{t\in [0,T]}|\langle \nu_s,\phi(\cdot)\rangle - \langle \nu_{\ell_N(s)}^N,\phi(\cdot)\rangle|\xrightarrow[N\to\infty]{}0$ for all $\phi\in \C{BC}^{2,1}$ $\B{P}$-a.s. By Proposition \ref{vagueContinuity} the map $t\mapsto \nu_t$ is continuous on $[0,T]$ $Q^\infty_T$-a.s., so we have that for $\phi\in \C{BC}^{2,1}$:
    $$|\langle \phi,\nu_s^N\rangle-\langle \phi,\nu_s\rangle|\leq |\langle \phi,\nu_s^N\rangle -\langle \phi, \nu_{\ell_N^{-1}(s)}\rangle|+|\langle \phi,\nu_{\ell_N^{-1}(s)}\rangle - \langle \phi,\nu_s\rangle|\xrightarrow[N\to\infty]{\B{P}-a.s.}0,$$
    which is to say that $\nu_s^N$ converges to $\nu_s$ for all $s\in [0,T]$ in the weak topology. Therefore for $\phi\in \C{BC}^{2,1}$, $\langle \nu^N_t,\phi(\cdot,t) \rangle \xrightarrow[N\to\infty]{\B{P}-a.s.} \langle \nu_t,\phi(\cdot,t)\rangle$ and, by Lemma \ref{bConvLemma} and the dominated convergence theorem,
    \begin{align*}
    \int_0^t \langle \nu_s^N, \frac12 \phi_{xx}(\cdot,s)+\phi_t(\cdot,s)+f(s)\phi(\cdot,s)\rangle ds &\xrightarrow[N\to\infty]{\B{P}-a.s.}\int_0^t \langle \nu_s, \frac12 \phi_{xx}(\cdot,s)+\phi_t(\cdot,s)+f(s)\phi(\cdot,s)\rangle ds,\\
    \int_0^t \frac{f(s)\langle \nu_s^N, \phi(\cdot,)\omega(\cdot-b^N(s))\rangle}{\langle \nu_s,\omega(\cdot-b^N(s))\rangle}ds &\xrightarrow[N\to\infty]{\B{P}-a.s.}\int_0^t \Lambda \langle \nu_s,\phi(\cdot,s)\omega(\cdot-b^\nu(s))\rangle ds.\end{align*}
    Finally, observe that $G(\nu^N,\phi,t)$ is bounded uniformly in $N$ by $\|\phi\|_\infty + t(\frac12\|\phi_{xx}\|_\infty+\|\phi_t\|_\infty+(\Lambda+\|f\|_\infty)\|\phi\|_\infty)$. Therefore $(|G(\nu^N,\phi,t)|)_{N\geq 0}$ is uniformly integrable and converges to $0$ in expectation, and hence converges to $0$ almost surely (see Appendixes, Proposition 2.3 in \cite{ethierKurtz}). Therefore $\lim_{N\to\infty}G(\mu^N,\phi,t)=0$ $\B{P}$-a.s. Since $\nu_s\ll\lambda$, thus $\langle \nu_s,\omega(\cdot-b^\nu(s))\rangle = f(s)/\Lambda$, and so:
    \begin{align*}
        \langle \mu_t,\phi(\cdot,t)-\langle \rho,\phi(\cdot,0)\rangle - \int_0^t \langle \mu_s,\frac12\phi_{xx}(\cdot,s)+\phi_t(\cdot,s)+f(s)\phi(\cdot,s)-\Lambda \phi(\cdot,s)\omega(\cdot-b(s))\rangle ds=0
    \end{align*}
    $\B{P}$-a.s. Under the measure $\B{P}$, $(\nu_s)_{s\in [0,T]}\sim Q_T^\infty$, therefore \eqref{convOfWeakRep} holds $Q_T^\infty$-almost surely, as required. 
\end{proof}

Finally we show that the equation \eqref{convOfWeakRep} has a unique solution when $\omega$ is continuous. In the case that $\omega(x)=\is_{x\leq 0}$, we already know by Theorem 1.8 of \cite{ettingerHeningEvans} that the weak solution is unique. 

\begin{propn} \label{uniquenessOmegaCts}
    Suppose that $\omega$ is continuous. Then there exists at most one measure-valued process $(\mu_t)_{t\in [0,T]}$ satsifying \eqref{convOfWeakRep} for all test functions $\phi\in \C{BC}^{2,1}(\B{R}\times [0,\infty),\B{R})$.
\end{propn}

\begin{proof}
    Let $(\mu_t)_{t\in [0,T]}$ and $(\nu_t)_{t\in [0,T]}$ both satisfy \eqref{convOfWeakRep} for all $\phi\in \C{BC}^{2,1}$, with $\mu_0=\rho=\nu_0$. Let $\psi\in C_b(\B{R})$ be a continuous test function with $\|\psi\|_\infty\leq 1$, and let $\phi(x,t)$ be the solution of the backwards heat equation $\phi_t + \frac12 \phi_{xx}=0$ with terminal condition $\phi(x,t)=\psi(x)$. So $|\phi(x,s)|\leq 1$ for all $x\in \B{R}$ and $s\in [0,t]$. Then \eqref{convOfWeakRep} yields that:
    \begin{align*}
        \langle \mu_t,\psi\rangle - \langle \nu_t,\psi\rangle = \int_0^t \langle \mu_s - \nu_s, f(s)\phi(\cdot,s)\rangle - \Lambda\langle \mu_s,\phi(\cdot,s)\omega(\cdot-b^\mu(s))\rangle + \Lambda \langle \nu_s,\phi(\cdot,s)\omega(\cdot-b^\nu(s))\rangle ds.
    \end{align*}
    Then by the triangle inequality
    \begin{align*}
        |\langle \nu_s,\phi(\cdot,s)\omega(\cdot-b^\nu(s))\rangle &- \langle \mu_s,\phi(\cdot,s)\omega(\cdot-b^\mu(s))\rangle| \\
        &\leq |\langle \nu_s - \mu_s, \phi(\cdot,s)\omega(\cdot-b^\mu(s))\rangle|+|\langle \nu_s,\phi(\cdot,s)(\omega(\cdot-b^\nu(s))-\omega(\cdot-b^\mu(s)))\rangle|.
    \end{align*}
    Now define the distance $d(\mu,\nu):=\sup_{\{g\in C(\B{R}):\|g\|_\infty \leq 1\}}|\langle \mu,g\rangle - \langle \nu,g\rangle|$ between two measures. So then as the function $\phi(\cdot,s)\omega(\cdot-b^\mu(s))$ is continuous in $x$ and bounded by $1$, thus $\|\langle \nu_s-\mu_s,\phi(\cdot,s)\omega(\cdot-b^\mu(s))\rangle|\leq d(\mu_s,\nu_s)$. Now suppose without loss of generality that $b^\mu(s)\leq b^\nu(s)$. Since $\omega$ is monotone decreasing, thus $\omega(x-b^\nu(s))-\omega(x-b^\mu(s))\geq 0$ for all $x$, so that:
    \begin{align*}
        |\langle \nu_s, \phi(\cdot,s)\big(\omega(\cdot-b^\nu(s))-\omega(\cdot-b^\mu(s)))\rangle| \leq \|\phi(\cdot,s)\|_\infty\langle \nu_s, \omega(\cdot-b^\nu(s))-\omega(\cdot-b^\mu(s))\rangle
    \end{align*}
    Next, we observe that by the continuity of $\omega$, $\langle \nu_s,\omega(\cdot-b^\nu(s))\rangle = f(s)/\Lambda = \langle \mu_s,\omega(\cdot-b^\mu(s))\rangle$, so that the right hand side above is 
    \begin{align*}\|\phi(\cdot,s)\|_\infty |\langle \mu_s,\omega(\cdot-b^\mu(s))\rangle - \langle \nu_s,\omega(\cdot-b^\mu(s))\rangle|\leq d(\mu_s,\nu_s).
    \end{align*}
    Putting this together, using the fact that $f$ is bounded by $\Lambda$, yields:
    \begin{align*}
        |\langle \mu_t,\psi\rangle-\langle \nu_t,\psi\rangle|\leq 3\Lambda\int_0^t d(\mu_s,\nu_s)ds
    \end{align*}
    for any continuous $\psi$ bounded by $1$. Therefore $d(\mu_t,\nu_t)\leq 3\Lambda \int_0^t d(\mu_s,\nu_s)$ and hence by Gr\"{o}nwall's inequality, $d(\mu_s,\nu_s)\equiv 0$, so that we can conclude $\mu_s\equiv \nu_s$, and thus the solution is unique. 
\end{proof}

\begin{proof}\textit{(Of Theorem \ref{1stMainThemOmega})}
By Proposition \ref{tightness}, the sequence $(Q^N_T)_{N=1,2,\ldots}$ has subsequential limits, and by Proposition \ref{weakRepPropn}, if $(\mu_t)_{t\in [0,T]}\sim Q^\infty_T$, then $\mu_t$ is a weak solution to the FBP \eqref{omegaNBMMHydro}. By Theorem 1.8 of \cite{ettingerHeningEvans}  (in the case that $\omega(x)=\is_{x\leq 0}$) and Proposition \ref{uniquenessOmegaCts} (in the case that $\omega$ is continuous), this solution is unique. Therefore as all subsequential limits are the same, $\mu_t$ is the unique weak solution to \eqref{omegaNBMMHydro}.
\end{proof}

\bibliographystyle{plain}
\bibliography{bps}

\end{document}